\documentclass[12pt,a4paper,psamsfonts]{amsart}
\usepackage{fancyhdr}
\usepackage{appendix}
\usepackage{amssymb,amscd,amsxtra,calc}
\usepackage{mathrsfs}
\usepackage{amsmath}
\usepackage{cmmib57}
\usepackage{multirow}
\usepackage{verbatim}
\usepackage[all]{xy}
\usepackage[colorlinks,linkcolor=blue,anchorcolor=blue,citecolor=green]{hyperref}
\theoremstyle{plain}
    \newtheorem{thm}{Theorem}[section]

     \newtheorem{defi}[thm]{Definition}
    
       \newtheorem{problem}[thm]{Problem}
    
    \newtheorem{lem}[thm]{Lemma}
    \newtheorem{pro}[thm]{Proposition}
    \newtheorem{question}[thm]{Question}
    
    \newtheorem{remark}[thm]{Remark}

\newtheorem{setup}[thm]{}

\makeatletter

\newcommand{\Rmnum}[1]{\expandafter\@slowromancap\romannumeral #1@}
\makeatother
\begin{document}

\bibliographystyle{plain}
\title[Threefolds with the action of an abelian group of maximal rank]
{Compact K\"ahler threefolds with the action of an abelian group of maximal rank
}

\author{Guolei Zhong}
\address
{
\textsc{Department of Mathematics} \endgraf
\textsc{National University of Singapore,
Singapore 119076, Republic of Singapore
}}
\email{zhongguolei@u.nus.edu}

\begin{abstract}
In this note, we study the normal compact K\"ahler (possibly singular) threefold  $X$  admitting the action of a free abelian group $G$ of maximal rank, all the non-trivial elements of which are of positive entropy. 
If such $X$ is further assumed to have only terminal singularities, then we prove that  it is either a rationally connected projective threefold or bimeromorphic to a quasi-\'etale quotient of a complex $3$-torus. 
\end{abstract}


\maketitle


\section{Introduction}
We work over the field $\mathbb{C}$ of complex numbers. 
Let $X$ be a compact K\"ahler manifold of dimension $n$. 
For an automorphism $g\in \textup{Aut}(X)$, the \textit{topological entropy} of $g$, defined in the theory of dynamical systems, turns out to coincide with the logarithm of the spectral radius of the pull-back operator $g^*$ acting on $\oplus_{0\le p\le n}H^{p,p}(X,\mathbb{R})$ (cf.~\cite{gromov2003on} and \cite{yomdin1987volume}).
We say that  $g\in\textup{Aut}(X)$ is of \textit{positive entropy}, if the topological entropy is strictly positive, or equivalently, the spectral radius of $g^*$ acting on $H^{p,p}(X,\mathbb{R})$ is strictly larger than $1$ for all or for some $p$ with $1\le p\le n-1$.
We say that $g\in\textup{Aut}(X)$ is of \textit{null entropy} if $g$ is not of positive entropy.
See the survey \cite{dinh2012tits} and the references therein.

T.-C. Dinh and N. Sibony proved in  \cite{dinh2004groupes} that every commutative subgroup $G\subseteq\textup{Aut}(X)$  has the (dynamical) rank $\le n-1$, if all the non-trivial elements of $G$ are of positive entropy. 
Subsequently, D.-Q. Zhang proved a theorem of Tits type for compact K\"ahler manifolds (cf.~\cite[Theorem 1.1]{zhang2009theorem}), extending the classical Tits alternative, and also extending \cite{dinh2004groupes} to the solvable case. 

Problem \ref{prob_main} below was first stated in \cite{dinh2004groupes} (cf.~\cite[Problem 1.5]{dinh2012tits}). 
The interest of studying these $X$ when $G$ has maximal  rank  is also the initial point of this paper.

\begin{problem}\label{prob_main}
\textup{Classify compact K\"ahler manifolds $X$ of dimension $n\ge 3$ admitting a free abelian group $G$ of automorphisms of rank $n-1$ which  is of positive entropy.}
\end{problem}

In the singular setting, we consider a normal projective variety (resp.\,a compact K\"ahler space with at worst rational singularities) $X$ and define the \textit{first dynamical degree} $d_1(f)$ of an automorphism $f\in\textup{Aut}(X)$ as the spectral radius $\rho(f^*)$ of its natural pull-back $f^*$ on the N\'eron-Severi group $\textup{N}^1(X)$ (resp. the Bott-Chern cohomology space $H_{\textup{BC}}^{1,1}(X)$).
We say that $f$ is of \textit{positive entropy} if $d_1(f)>1$, otherwise it is of \textit{null entropy}. 
Note that our definition here for $d_1(f)$ in the singular setting coincides with the usual one when $X$ is smooth:
\begin{remark}\label{cm_rmk_bimero_invariant}
\textup{(1). For an $n$-dimensional normal projective variety $X$ with an automorphism $f\in\textup{Aut}(X)$, we take the $f$-equivariant resolution $\pi:\widetilde{X}\to X$ (over a field of characteristic zero; cf.\,e.g.\,\cite[Theorem 2.0.1]{wlodarczk2009resolution}) and denote by $\widetilde{f}$ the lifting of $f$ to $\widetilde{X}$. 
Taking an ample divisor $H$ on $X$, we  have
\begin{align*}
d_1(f)=\rho(f^*|_{\textup{N}^1(X)})&=\lim_{m\to\infty}((f^m)^*H\cdot H^{n-1})^{1/m}\\
&=\lim_{m\to\infty}(\pi^*(f^m)^*H\cdot (\pi^*H)^{n-1})^{1/m}\\
&=\lim_{m\to\infty}((\widetilde{f}^m)^*\pi^*H\cdot (\pi^*H)^{n-1})^{1/m}=\rho(\widetilde{f}^*|_{\textup{N}^1(\widetilde{X})})=d_1(\widetilde{f}).
\end{align*}
Here, the second and the fifth equalities are due to \cite[Proposition A.2]{nakayama2009building}, noting that $\pi^*H$ is a nef and big divisor on $\widetilde{X}$ and lies in the interior of the pseudo-effective cone which spans $\textup{N}^1(\widetilde{X})$.}

\textup{(2). For a compact K\"ahler space $X$ with at worst rational singularities, if $f\in\textup{Aut}(X)$ is an automorphism, then similarly, we take the $f$-equivariant resolution $\widetilde{X}\to X$ and lift $f$ to $\widetilde{f}$. 
Replacing $\textup{N}^1(X)$ in (1) by the Bott-Chern cohomology $\textup{H}^{1,1}_{\textup{BC}}(X)$ (cf.~\cite[Definition 4.6.2]{boucksom2013an}) and replacing the ample divisor $H$ in (1) by a K\"ahler class $[\omega]$ on $X$, we see that $d_1(f)$ coincides with $d_1(\widetilde{f})$ by applying \cite[Remark 2.3,  Propositions  2.6 and 2.8]{zhong2019intamplified}.}

\textup{(3). In the published version of this paper, we removed this Remark \ref{cm_rmk_bimero_invariant} therein for the organization.}
\end{remark}

In the past decade, D.-Q. Zhang established the $G$-equivariant minimal model program  for projective varieties, and Problem \ref{prob_main}  has thus been intensively studied  
 in his series papers (cf.~\cite{zhang2009theorem}, \cite{zhang2013algebraic} and \cite{zhang2016ndimension}). 
Inspired by \cite[Theorem 1.1]{zhang2016ndimension}, we are interested in
Problem \ref{prob_main} itself (without assuming the projectivity) and  ask the following question.

\begin{question}\label{ques_1.4}
Let $X$ be a normal compact K\"ahler space of dimension $n$ with mild singularities.
Suppose that $\textup{Aut}(X)\supseteq G:=\mathbb{Z}^{\oplus n-1}$ and every non-trivial  element of $G$ is of positive entropy.
Is $X$ either  rationally connected, or $G$-equivariantly bimeromorphic to a $Q$-torus?
\end{question}

Recall that a normal compact K\"ahler space $X$ is said to be a \textit{Q-torus}, if there exists a complex torus (full rank) $T$ and a finite surjective morphism $\pi:T\to X$ such that $\pi$ is \'etale in codimension $1$ (cf.~\cite[Definition 2.13]{nakayama2010polarized}).

Question \ref{ques_1.4} is related to the  building blocks of the commutative subgroup $G$  on   K\"ahler spaces. 
For example, we don't know whether the automorphisms of Calabi-Yau manifolds would appear as a fundamental building block of $G$ in the general non-algebraic situation.
Question \ref{ques_1.4} also seems to have its own interest from the viewpoints of differential geometry, since the minimal model program (MMP for short) for  K\"ahler spaces is less known in higher dimension (cf.~\cite{Horing2015mori}, \cite{horing2016minimal} and \cite{das2020the} for the MMP for K\"ahler threefolds).

In this paper, we shall study Problem \ref{prob_main} and  answer Question \ref{ques_1.4} affirmatively for K\"ahler threefolds, with
the main result Theorem \ref{thm1.1} (cf.~Remark \ref{rmk_solvable} for the solvable case). 
Let us consider the following hypothesis:

\setlength\parskip{4pt}\par \vskip 0.3pc \noindent
\textbf{Hyp\,(A): $X$ is a normal $\mathbb{Q}$-factorial compact K\"ahler threefold with at worst terminal singularities; $\textup{Aut}(X)\supseteq G:=\mathbb{Z}^2$ and every non-trivial element of $G$ is of positive entropy.}

 \setlength\parskip{0pt}
 \par \vskip 0pc \noindent

\begin{thm}\label{thm1.1}
Suppose  $(X,G)$ satisfies \textbf{Hyp\,(A)}. Then the following assertions hold:
\begin{enumerate}
\item[\textup{(1)}] If $K_X$ is not pseudo-effective, then $X$ is rationally connected. In particular, $X$ is a projective threefold.
\item[\textup{(2)}] If $K_X$ is pseudo-effective, then the Kodaira dimension $\kappa(X)=0$ and $X$ is bimeromorphic to $X_{\textup{min}}:=T/H$ for a finite group $H$ acting freely outside a finite set of a complex $3$-torus $T$. Moreover, $G$ descends to an automorphism subgroup $G|_{X_{\textup{min}}}\subseteq\textup{Aut}(X_{\textup{min}})$ and  further lifts to $G_T\subseteq \textup{Aut}(T)$ such that $G|_{X_{\textup{min}}}\cong G_T/H$.
\end{enumerate}
\end{thm}

To prove Theorem \ref{thm1.1}, the main obstacle is to deal with the case  $\kappa(X)=0$ (cf.~Proposition \ref{prop_kappa<=0}).
Now, we briefly describe our strategy.

\begin{setup}[\textbf{Strategy for the proof of Theorem \ref{thm1.1} when $\kappa(X)=0$}]
\textup{First,  we reduce $X$ to its minimal model $X_{\textup{min}}$ equipped with the induced action  $G|_{X_{\textup{min}}}$ (of the same rank) by  bimeromorphic transformations. 
    Note that  
    elements in  $G|_{X_{\textup{min}}}$ may not be biholomorphic.} 

\textup{Second, we 
study the Albanese closure $X'_{\textup{min}}$ of $X_{\textup{min}}$ (cf.~Definition \ref{defi-alb-closure}) and classify $X'_{\textup{min}}$ (cf.~Proposition \ref{keyproposition}). 
Meanwhile, we lift $G|_{X_{\textup{min}}}$ to its Albanese closure $X'_{\textup{min}}$  and get a subgroup $G|_{X'_{\textup{min}}}\subseteq\textup{Bim}(X'_{\textup{min}})$  (of the same rank)  on  $X'_{\textup{min}}$. }

\textup{Finally, we lift $G|_{X'_{\textup{min}}}$ to its splitting cover $E\times F$, where $E$ is a torus and $F$ is a weak Calabi-Yau space (cf.~Definition \ref{defi-wcy}). 
We  deduce  that the group to which $G$ lifts    actually consists of biholomorphic transformations (in $\textup{Aut}(E\times F)$). 
Then, we show that there exists a finite index subgroup $G_1\subseteq G|_{E\times F}$ such that  $(E\times F, G_1)$ satisfies \textbf{Hyp\,(A)}. 
 So we conclude our theorem by applying the inspiring result \cite[Lemma 2.10]{zhang2009theorem}.} 
\end{setup}

We are concerned with the following special case of Theorem \ref{thm1.1}  when  $X$ is minimal.

\begin{pro}\label{thm1.2}
Suppose that $(X,G)$ satisfies \textbf{Hyp\,(A)}. 
Suppose further that $X$ is minimal, i.e., the canonical divisor $K_X$ is nef.
Then, $X$ is a quasi-\'etale quotient of a complex torus $T$. Further, the group $G$ lifts to $G_T\subseteq\textup{Aut}(T)$ on $T$. 
\end{pro}

\begin{remark}[Differences with earlier papers]\label{reamrk1.3}
\textup{Initially, we tried to run the $G$-equivariant minimal model program ($G$-MMP for short) for $X$ as in \cite{zhang2016ndimension}, but we met  problems about the finiteness of $(K_X+\xi)$-negative extremal rays. Here, $\xi=\sum\xi_i$ is the sum of  common nef eigenvectors of $G$ (cf.~\cite[Theorems 4.3 and 4.7]{dinh2004groupes}). 
Besides,  for a (not necessarily rational) nef and big class $\eta$ on a compact K\"ahler manifold, little was known about  Kodaira's Lemma (cf. e.g. \cite[Lemma 2.60]{kollar2008birational}) for $\eta$ and the classical  base-point-free  theorem (cf.~\cite[Conjecture 1.6 and Theorem 1.7]{das2020the}).  
Due to these difficulties mentioned,  we will not (and cannot) run the $G$-MMP for $X$. 
}

\textup{Thanks to the particularity of terminal threefolds, we can analyze the quasi-\'etale covers of its minimal model $X_{\textup{min}}$ (cf.~Proposition \ref{pro1}) to overcome the difficulty. 
We lift every bimeromorphic transformation of $X_{\textup{min}}$ to its quasi-\'etale cover and use the property of minimal surfaces to show that every  lifted bimeromorphic transformation 
is indeed an automorphism (cf.~Lemma \ref{splitting lemma}).
Applying the same strategy for the proof of Theorem \ref{thm1.1}, we can finally extend our main result to the solvable case  (cf.~Remark \ref{rmk_solvable}).}
\end{remark}

As an application, we refer to \cite[Example 4.5]{dinh2004groupes}  for concrete examples with respect to Theorem \ref{thm1.1}.
We  note that there exists a $3$-dimensional $Q$-torus $X$ being rational, which also admits  $\mathbb{Z}^2\cong G\subseteq \textup{Aut}(X)$ of positive entropy, though in this case, such $X$ has worse singularities (cf.~\cite[Example 1.6]{zhang2016ndimension} and \cite[Theorem 5.9]{oguiso2014some}). 

In comparison, we pose the following question at the end of this section. 
\begin{question}
Does there exist a pair $(X,G)$ satisfying \textbf{Hyp\,(A)} such that $X$ is not a $Q$-torus even after bimeromorphic change of models? 
\end{question}

\subsubsection*{\textbf{\textup{Acknowledgments}}}
The author would like to thank Professor De-Qi Zhang for many inspiring  discussions and encouragements for this project. 
He thanks Professor Tien-Cuong Dinh for pointing out Remark \ref{cm_rmk_bimero_invariant}, and the referees for very careful reading and many  suggestions to improve the paper.
The author is supported by a President's Graduate Scholarship of NUS.

\section{Preliminaries and the proof of Proposition \ref{thm1.2}}
Let $X$ be a normal compact K\"ahler space. 
We refer to \cite[Chapter 2]{kollar2008birational} for different kinds of singularities.   
Let $\textup{H}^{1,1}_{\textup{BC}}(X)$ be the \textit{Bott-Chern cohomology} (cf.~\cite[Definition 4.6.2]{boucksom2013an}) and $\textup{N}_1(X)$ the  space of real closed currents of bidimension $(1,1)$ modulo the equivalence relation: $T_1\equiv T_2$ if and only if $T_1(\eta)=T_2(\eta)$ for all real closed $(1,1)$-forms $\eta$ with local potentials (cf.~\cite[Section 2]{horing2016minimal}).

Let $f:X\rightarrow Y$ be a surjective morphism (i.e., a holomorphic map) between normal compact complex spaces. The morphism $f$ is said to be \textit{finite} (resp.~\textit{generically finite}) if $f$ is proper and has discrete fibres (resp.~proper and finite outside a nowhere dense analytic closed subspace of $Y$). 
We say that $f$ is quasi-\'etale if $f$ is finite and \'etale in codimension one.
Denote by $\textup{Bim}(X)$ (resp.~\textup{Aut}(X)) the group of bimeromorphic  (resp.~biholomorphic) self-maps of $X$. Also, $\textup{Bim}(X)$ (resp.~$\textup{Aut}(X)$) is called the \textit{bimeromorphic transformation} (resp.~\textit{analytic automorphism}) \textit{group} of $X$. 

 
We shall apply the following theorem (cf.~\cite[Theorem 1]{kawamata2008flops}) in Section \ref{section3} and recall it here for the convenience of readers. 
Note that the original proof (in the projective setting) still works in our present case for compact K\"ahler (terminal) threefolds.
\begin{thm}[\cite{kawamata2008flops}]\label{lemisomorphic}
Let $\tau: X\dashrightarrow X'$ be a bimeromorphic map between compact K\"ahler threefolds with at worst terminal singularities. 
Suppose  $X$ and $X'$ are both minimal, i.e., the canonical divisors $K_X$ and $K_{X'}$ are nef. 
Then $\tau$ is isomorphic in codimension one.
\end{thm}

Let $X$ be a normal compact K\"ahler space.
We define the \textit{augmented irregularity} $q^\circ(X)$ of $X$ to be the supremum of the irregularities $q(X')$ when $X'$ runs over  the finite  quasi-\'etale covers $X'\rightarrow X$ (cf.~\cite[Section 4]{nakayama2009building} and \cite[Definition 2.4]{graf2018algebraic}). 
Here, the irregularity $q(X')$ is defined to be $q(\widetilde{X}')=h^1(\widetilde{X}',\mathcal{O}_{\widetilde{X}'})$, where $\widetilde{X}'\rightarrow X'$ is a resolution. 
In general, the augmented irregularity can be infinite.
For example, one can consider any finite \'etale cover of genus $\ge 2$ curves and apply  Riemann-Hurwitz formula.

Now, we define the weak Calabi-Yau space in the K\"ahler case  (cf.~\cite[Definition 2.9]{nakayama2010polarized}).
\begin{defi}\label{defi-wcy}
\textup{A normal compact K\"ahler space $X$ is said to be a \textit{weak Calabi-Yau space} if $X$ has at worst canonical singularities, $K_X\sim_\mathbb{Q} 0$ and $q^{\circ}(X)=0$. }
\end{defi}

We are  interested in the compact K\"ahler threefolds with only canonical singularities and  vanishing first Chern class (e.g., a weak Calabi-Yau threefold).
The following  theorem gives a nice structure of the Albanese map in this situation (cf.~\cite[Theorem 1.10]{graf2018algebraic}):
 \begin{thm}[cf.~\cite{graf2018algebraic}]\label{thm-graf-thm1.10}
 Let $X$ be a normal compact K\"ahler threefold with only canonical singularities and vanishing first Chern class $c_1(X)=0$.  
 Let $\alpha:X\to A:=\textup{Alb}(X)$ be the Albanese map.
 Then there exists a finite \'etale cover $A_1\to A$ such that $X\times_A A_1$ is isomorphic to $F\times A_1$ over $A_1$, where $F$ is connected.
 In particular, $\alpha$ is a surjective analytic fibre bundle with connected fibres.
 \end{thm}


Applying Theorem \ref{thm-graf-thm1.10}, we classify the normal compact K\"ahler threefold  with  canonical singularities and a trivial canonical class. 
Indeed, 
one can further follow the idea of \cite[Proposition 2.10]{nakayama2010polarized} to generalize the classification to the case of klt singularities. 
 \begin{pro}\label{keyproposition}
Let $X$ be a normal compact K\"ahler threefold with at worst canonical singularities such that $K_X\sim_{\mathbb{Q}}0$. Then:
\begin{enumerate}
\item[\textup{(1)}] $q(X)\le q^\circ(X)\le 3$ and $q^\circ(X)\neq 2$. In particular, there is a quasi-\'etale Galois cover $X'\rightarrow X$ such that $q(X')=q^\circ(X)$.
\item[\textup{(2)}] If $q^\circ(X)=3$, then $X$ is a quasi-\'etale quotient of a complex torus.
\item[\textup{(3)}] If $q^\circ(X)=0$, then $X$ is projective, i.e., a weak Calabi-Yau threefold is projective.\item[\textup{(4)}] If $q^\circ(X)=1$, then there exists a quasi-\'etale cover $E\times F\rightarrow X$, where $E$ is an elliptic curve and $F$ has at worst canonical singularities such that $q^\circ(F)=0$ and its minimal resolution is a K3 surface.
\end{enumerate}
\end{pro}

\begin{proof}
We first show (1). For each quasi-\'etale cover $X'\rightarrow X$, our $K_{X'}\sim_{\mathbb{Q}} 0$ and $X'$ also has only canonical  singularities (cf.~\cite[Proposition 5.20]{kollar2008birational}). 
Then the Albanese map $X'\rightarrow \textup{Alb}(X')$ is surjective holomorphic (cf.~\cite[Theorem 5.22]{kollar2008birational} and Theorem \ref{thm-graf-thm1.10}), hence $q(X')\le\dim X'=3$. So $q(X)\le q^\circ(X)\le 3$ and by the boundedness of $q^\circ(X)$, there exists a quasi-\'etale cover $X'\rightarrow X$ such that $q(X')=q^\circ(X)$. 
Taking the Galois closure $X''\rightarrow X$ of $X'\rightarrow X$ with the induced Galois cover $X''\rightarrow X'$, we get $q^\circ(X)=q(X')\le q(X'')\le q^\circ(X)$. 
Hence, $X''\rightarrow X$ is the cover we are looking for. 

Suppose that $q^\circ(X)=2$. 
Then, there exists a quasi-\'etale Galois cover $X''\rightarrow X$ such that $q^\circ(X)=q(X'')=2$. 
By Theorem \ref{thm-graf-thm1.10} and the adjunction, the Albanese map $\alpha:X''\rightarrow \textup{Alb}(X'')$ is a fibre bundle with the smooth fibre $F$ being an elliptic curve. 
After an \'etale base change $A_1\rightarrow \textup{Alb}(X'')$, there is an \'etale cover $F\times A_1\rightarrow X''$ (cf.~Theorem \ref{thm-graf-thm1.10})  with $q(F\times A_1)=3$, a contradiction to $q^\circ(X)=2$.
So (1) is proved.

(2) follows from Theorem \ref{thm-graf-thm1.10} and (3) follows from  \cite[Remark 2.6]{graf2018algebraic}. 
Now, we prove (4). 
Since $q^\circ(X)=1$, there exists a quasi-\'etale Galois cover $X_1\rightarrow X$ such that $q(X_1)=q^\circ(X)=1$. 
Let $F$ be a fibre of $\alpha: X_1\rightarrow  \textup{Alb}(X_1)$.  
By Theorem \ref{thm-graf-thm1.10}, there exists an \'etale cover $T_1\rightarrow \textup{Alb}(X_1)$ such that $T_1\times F\rightarrow X_1$ is \'etale. 
Then $T_1\times F$ has only canonical singularities (cf.~\cite[Proposition 5.20]{kollar2008birational}). 
In view of the commutative diagram of the resolutions of $T_1\times F$ and $F$, our  $F$ has only canonical singularities. 
Further, $K_{T_1\times F}\sim_{\mathbb{Q}}0$ implies that  $K_F\sim_{\mathbb{Q}}0$, hence $F$ is a surface with at worst canonical singularities such that $K_F\sim_{\mathbb{Q}}0$.
Since $q^\circ(X)=1$, we have $q(F)=q^\circ(F)=0$. So with $F$  replaced by its global index-one cover (such that $K_F\sim 0$), 
the minimal resolution of $F$ is a K3 surface. 
\end{proof}


In what follows, we  discuss the Albanese closure in codimension one of a normal compact K\"ahler threefold (cf.~\cite[Section 2]{nakayama2010polarized} and \cite[Proposition 5.1]{zhong2019intamplified}).  
\begin{defi}[Albanese closure]\label{defi-alb-closure}
\textup{Let $X$ be a normal compact K\"ahler threefold with only  canonical singularities such that $K_X\sim_{\mathbb{Q}}0$. We say that the finite morphism $\tau:\widetilde{X}\rightarrow X$ is the \textit{Albanese closure in codimension one} of $X$ 
if the following assertions  hold.
\begin{enumerate}
\item[\textup{(1)}] $\tau$ is quasi-\'etale;
\item[\textup{(2)}] $q^\circ(X)=q(\widetilde{X})$;
\item[\textup{(3)}] $\tau$ is Galois; and
\item[\textup{(4)}] For any other finite morphism $\tau':X'\rightarrow X$ satisfying the conditions (1) and (2), there exists a quasi-\'etale morphism $\sigma:X'\rightarrow \widetilde{X}$ such that $\tau'=\tau\circ\sigma$.
\end{enumerate}}
\end{defi}

Now, we prove the existence of Albanese closure in the spirit of \cite[Lemma 2.12]{nakayama2010polarized} (cf.~\cite[Proposition 4.3]{nakayama2009building}) for readers' convenience. 

\begin{pro}\label{pro1}
Let $X$ be a normal compact K\"ahler threefold with at worst canonical singularities such that $K_X\sim_\mathbb{Q}0$. Then the Albanese closure $\tau:\widetilde{X}\rightarrow X$ exists.
\end{pro}

\begin{proof}
If $q^\circ(X)=0$, then $X$ is projective (cf.~Proposition \ref{keyproposition} (3)) and our result follows immediately from \cite[Lemma 2.12]{nakayama2010polarized}. So we may assume $q^\circ(X)>0$. By Proposition \ref{keyproposition}, there exists a quasi-\'etale Galois cover $X_1\rightarrow X$ with the Galois group $G_1$ such that $q(X_1)=q^\circ(X)$. Then $K_{X_1}\sim_{\mathbb{Q}}0$ and $X_1$ also has at worst canonical singularities. Note that $G_1=\textup{Gal}(X_1/X)$ acts naturally on the Albanese torus $\textup{Alb}(X_1)$ (by the universal property of Albanese torus) and hence induces an action on $\textup{H}_1(\textup{Alb}(X_1),\mathbb{Z})$. If we write $\textup{Alb}(X_1):=V/\Lambda$ for a vector space $V$ and a lattice $\Lambda$, then $\textup{H}_1(\textup{Alb}(X_1),\mathbb{Z})\cong \Lambda$. Hence, there exists a natural homomorphism $\varphi: G_1\rightarrow \textup{Aut}(\Lambda)$.

 Let $H_1:=\ker\varphi$ and $\widetilde{X}:=X_1/H_1$. Then $\textup{Alb}(X_1)/H_1$ is  a torus, since  $H_1$ acts on $\textup{Alb}(X_1)$ as translations. Then, $\textup{Alb}(X_1)/H_1$ is the Albanese torus of $\widetilde{X}$ by the universality, hence $q(\widetilde{X})=q(X_1)=q^\circ(X)$. Since $H_1$ is normal in $G_1$,  $\widetilde{X}\rightarrow X$ is Galois with the induced Galois group $G_1/H_1$. 
Therefore,  $\widetilde{X}\rightarrow X$ satisfies conditions (1), (2) and (3).

Now, let $X'\rightarrow X$ be an arbitrary quasi-\'etale cover such that $q(X')=q^\circ(X)$ and take the Galois closure $X_2\rightarrow X$ of $X_1\times_X X'\to X$ with the induced Galois covers $X_2\rightarrow X'$ and $X_2\rightarrow X_1$. 
Let $G_2:=\textup{Gal}(X_2/X)$ (resp.~$G_2':=\textup{Gal}(X_2/X')$) be the Galois group of $X_2\rightarrow X$ (resp.~$X_2\rightarrow X'$). 
Then, with the same argument as above, we  denote by $H_2\subseteq G_2$ the kernel of $G_2\rightarrow \textup{Aut}(\textup{H}_1(X_2,\mathbb{Z}))$. Note that $\textup{Alb}(X_2)\rightarrow \textup{Alb}(X_1)$ is an isogeny by choosing a proper origin on $\textup{Alb}(X_1)$. Thus, $H_2$ is the pull-back of $H_1\subseteq G_1$ under $X_2\rightarrow X_1$ and then $X_2/H_2\cong X_1/H_1=\widetilde{X}$. Moreover, $G_2'$ acts on $\textup{H}_1(\textup{Alb}(X_2),\mathbb{Z})$ trivially by the choice of $X'$ (so that $q(X')=q^\circ(X)$). As a result, $G_2'\subseteq H_2$ and we have a natural factorization $X'\rightarrow \widetilde{X}\rightarrow X$.
So (4) is satisfied.
\end{proof}

In the following, we recall the common nef eigenvectors of an automorphism subgroup. Actually, this part will not be used in proving Theorem \ref{thm1.1} and Proposition \ref{thm1.2}. However, we still formulate the extended Proposition \ref{pro2.8} to compare with earlier results. 

We follow \cite[the proof of Theorem 1.2]{zhang2013algebraic} to find the nef classes $\xi$ which are common eigenvectors of $G$ in our present case.
Suppose  $(X,G)$ satisfies \textbf{Hyp\,(A)}. Let $\pi:\widetilde{X}\rightarrow X$ be a $G$-equivariant resolution (cf.~\cite[Theorem 2.0.1]{wlodarczk2009resolution}). Applying the proof of \cite[Theorems 4.3 and 4.7]{dinh2004groupes} to the action of $G$ on the pullback $\pi^*\textup{Nef}(X)$ of the nef cone of $X$, we get nef classes $\pi^*\xi_i~(1\le i\le 3)$ on $\widetilde{X}$ as common eigenvectors of $G$ such that  the cup product $\xi_1\cup\xi_2\cup\xi_3\neq 0$, where $\xi_i$ are nef classes on $X$. 
Since $\xi_i~(1\le i\le 3)$ are eigenvectors, for each $g\in G$, we can write $g^*\xi_i=\chi_i(g)\xi_i$ with the characters $\chi_i:G\to\mathbb{R}_{>0}$. 
By the projection formula, our $\chi_1\chi_2\chi_3=1$. 
Let 
\begin{equation}\label{eq1}
	\xi:=\xi_1+\xi_2+\xi_3.
\end{equation}
 Then $\xi^3\ge \xi_1\cup\xi_2\cup\xi_3>0$, which implies that $\xi$ is a nef and big class on $X$.

We borrow the following results from \cite[Lemma 3.7]{zhang2016ndimension} and the proofs therein can be adapted into our present version. Readers may refer to \cite[Section 5]{graf2019finite} for the information of Chern classes on singular spaces.

\begin{lem}\label{lemperiodic}
Suppose that $(X,G)$ satisfies \textbf{Hyp\,(A)}. 
Then, for the $\xi$ in Equation (\ref{eq1}), the following assertions hold.
\begin{enumerate}
\item[\textup{(1)}] For every $G$-periodic $(k,k)$-class $\eta$ with $k=1$ or $2$, $\xi^{3-k}\cdot \eta=0$; in particular, $\xi^2\cdot c_1(X)=\xi\cdot c_1(X)^2=0$. 
\item[\textup{(2)}]	Let $Z\subseteq X$ be a $G$-periodic positive-dimensional proper subvariety of $X$. Then $\xi^{\dim Z}\cdot Z=0$.
\end{enumerate}
\end{lem}

The lemma below is a direct consequence of \cite[Th\'eor\`eme 3.1 and Corollaire 3.4]{dinh2004groupes}. 
It is also known as a generalization of Hodge-Riemann Theorem.

\begin{lem}\label{hodgeriemann}
Let $X$ be a normal compact K\"ahler threefold, $u_i\in \textup{H}^{1,1}_{\textup{BC}}(X)~(i=1, 2)$ nef classes and $\eta\in \textup{H}^{1,1}_{\textup{BC}}(X)$ such that $u_1\cup u_2\cup\eta=0$. Then we have:
\begin{enumerate}
\item[\textup{(1)}] $u_1\cup\eta^2\le 0$.
\item[\textup{(2)}] Suppose that $u_1=u_2$ is nef and big, and $\eta$ is nef. Then $u_1\cup\eta^2=0$ holds if and only if $\eta\equiv 0$, i.e., $\eta\cdot \Gamma=0$ for each $\Gamma\in\textup{N}_1(X)$. 
\end{enumerate}
\end{lem}
\begin{proof}
The first statement is easy to see by pulling back $u_i$ and $\eta$ to a smooth model $\widetilde{X}$ of $X$ (cf.~\cite[Proposition 2.6]{zhong2019intamplified} and \cite[Corollaire 3.4]{dinh2004groupes}). 
For the second statement, note that for every big class $u$, there exists a suitable resolution $\pi:\widetilde{X}\rightarrow X$ and an effective $\mathbb{R}$-divisor $E$ on $\widetilde{X}$ supported in the exceptional locus of $\pi$  such that $\pi^*u-E$ is a K\"ahler  class on $\widetilde{X}$ (cf.~\cite[Proposition 2.6]{zhong2019intamplified} and \cite[Proof of Theorem 3.1]{boucksom2004pseudo}). 
By the projection formula,  $(\pi^*u_1)^2\cup\pi^*\eta=\pi^*u_1\cup(\pi^*\eta)^2=0$. 
So it is easy to verify that $(\pi^*u_1-E)^2\cup\pi^*\eta=0$ and $(\pi^*u_1-E)\cup(\pi^*\eta)^2=0$, noting that  $\pi^*u_1-E$ and $\pi^*\eta$ are nef and $E\ge 0$.
Then, we apply \cite[Th\'eor\`eme 3.1]{dinh2004groupes} for $\pi^*u_1-E$ and $\pi^*\eta$ to conclude that $\pi^*\eta\equiv 0$. Hence $\eta\equiv 0$ by the projection formula (cf.~\cite[Proposition 3.14]{horing2016minimal}). 
\end{proof}

  The following proposition gives a sufficient condition for $X$ to be a $Q$-torus under the assumption of \textbf{Hyp\,(A)}. 
Note that, $\xi$ being K\"ahler implies that every $G$-periodic proper subvariety of $X$ is a single point (cf.~Lemma \ref{lemperiodic}).

\begin{pro}\label{pro2.8}
Suppose $(X,G)$ satisfies \textbf{Hyp\,(A)}. Suppose further that $X$ is minimal (i.e., $K_X$ is nef) and the $\xi$ in Equation (\ref{eq1}) is a K\"ahler class on $X$.
Then, $X\cong T/F$, where $T$ is a complex torus and $F$ is a finite group whose action on $T$ is free outside a finite subset of $T$. Further, the  group $G$ (acting on $X$) lifts to $G_T\subseteq \textup{Aut}(T)$ on $T$.
\end{pro}

\begin{proof}
	By Lemma \ref{lemperiodic}, $\xi^2\cdot c_1(X)=\xi\cdot c_1(X)^2=0$. 
Since  $\xi$ is  K\"ahler,  $c_1(X)=0$ (cf.~Lemma \ref{hodgeriemann}). 
Then, it follows from the abundance that $K_X\sim_{\mathbb{Q}} 0$ (cf.~\cite[Theorem 1.1]{campana2016abundance}). 
	Let $m$ be the minimal integer such that $m K_X\sim 0$. Taking the associated global index-one cover   $f:X':=\textbf{Spec}\oplus_{i=0}^{m-1}\mathcal{O}_X(-iK_X)\rightarrow X$, we can lift $G$ to $X'$ by their actions on $K_X$.

Now, $K_{X'}\sim 0$ and $X$ has only Gorenstein terminal singularities (cf.~\cite[Proposition 5.20]{kollar2008birational}).
	Let $\sigma:\widetilde{X}\rightarrow X'$ be a $G$-equivariant resolution minimal in codimension two (cf.~\cite[Definition 2.1 and Proposition/Definition 5.3]{graf2019finite}), and $c_2(X')$ the ``Birational'' second Chern class, which is defined as an element $c_2(X')\in \textup{H}^2(X',\mathbb{R})^\vee$ such that
	$$c_2(X')\cdot a:=\int_{\widetilde{X}}c_2(\widetilde{X})\cup \sigma^*(a),$$
for any $a\in \textup{H}^2(X',\mathbb{R})$. 
Here, $c_2(\widetilde{X})$ is the usual second Chern class of  $\widetilde{X}$. 

Since $f$ is finite, $\xi':=f^*\xi$ is a K\"ahler class on $X'$ (cf.~\cite[Proposition 3.5]{graf2019finite}).
We aim to show that $\xi'\cdot c_2(X')=0$. 
Indeed, for each summand $\xi_i':=f^*\xi_i$ of $\xi'$, we  choose $g_i\in G$ such that $g_i^*\xi_i'=\alpha_i\xi_i'$ with $\alpha_i>1$ (cf.~\cite[Lemma 3.7]{zhang2016ndimension}). 
Further, by \cite[Lemma 5.6]{graf2019finite}, $c_2(X')\cdot \xi_i'=0$ and thus $c_2(X')\cdot \xi'=0$. 
Applying \cite[Corollary 1.2]{graf2019finite}, we have $X\cong T/F$, where $T$ is a  complex torus and $F$ is a finite group whose action on $T$ is free outside a finite subset of $T$. 
Finally, according to \cite[\S 2.14]{zhang2013algebraic}, $G$ lifts to $G_T\subseteq \textup{Aut}(T)$.
\end{proof}


Now we come back to the proof of Proposition \ref{thm1.2}. 
First, we recall the following useful  result which is a special case of \cite[Lemma 2.10]{zhang2009theorem}. 
It will be crucially applied later on.

\begin{lem}[cf.~\cite{zhang2009theorem}]\label{lem-zhang-2-10}
Suppose that $(X,G)$ satisfies \textbf{Hyp\,(A)}.
Then there does not exist any $G$-equivariant surjective holomorphic map $\pi:X\to Y$ such that $\dim X>\dim Y>0$ and $G$ descends to a biregular action on $Y$.
\end{lem}

As an application of Lemma \ref{lem-zhang-2-10}, we are able to show the  proposition below.
\begin{pro}\label{prop_kappa<=0}
Suppose that $(X,G)$ satisfies \textbf{Hyp\,(A)}.
Then the Kodaira dimension  $\kappa(X)\le 0$.
In particular, either $\kappa(X)=0$ or $X$ is uniruled.
\end{pro}
\begin{proof}
Since $\textup{Bim}(X)$ is an infinite group, 
 $\kappa(X)<3$ (cf.~\cite[Corollary 14.3]{ueno2006classification}). 
Suppose that $1\le\kappa(X)<3$. 
Then, there exists an Iitaka fibration $\varphi:X\dashrightarrow \mathbb{P}^n$ such that $\dim \textup{Im}\,\varphi=\kappa(X)$. Taking a $G$-equivariant resolution $\widetilde{X}\to X$ (cf.~\cite[Section 1.4]{nakayama2009building} and \cite[Theorem 2.0.1]{wlodarczk2009resolution}), we may assume that $X$ is a compact K\"ahler smooth threefold. 
Besides, resolving the indeterminacy locus of $\varphi$, we can further assume that $\varphi$ is holomorphic (cf.~\cite[Theorem A and the remark therein]{nakayama2009building}). 
Then, we get a non-trivial $G$-equivariant fibration $X\rightarrow \textup{Im}\,\varphi$ such that $\dim\textup{Im}\,\varphi<3$, a contradiction to Lemma \ref{lem-zhang-2-10}. Therefore, under the assumption of \textbf{Hyp\,(A)}, $\kappa(X)\le 0$. 	

If $K_X$ is pseudo-effective, then applying \cite[Theorem 1.1]{horing2016minimal} and \cite[Theorem 1.1]{campana2016abundance}, we have $\kappa(X)=0$.
If $K_X$ is not pseudo-effective, then $X$ is uniruled (cf.~\cite{boucksom2004pseudo}).
\end{proof}


\begin{lem}\label{smallcase1}
Suppose  $(X,G)$ satisfies \textbf{Hyp\,(A)} and  $K_X\equiv 0$. 
Then $X$ is a $Q$-torus, i.e., a quasi-\'etale quotient of a torus $T$.
Moreover, the group $G$ lifts to $G_T\subseteq\textup{Aut}(T)$.
\end{lem}

\begin{proof}
By abundance (cf.~\cite[Theorem 1.1]{campana2016abundance}), $K_X\sim_{\mathbb{Q}} 0$. Taking the global index-one cover $\widetilde{X}\rightarrow X$, we see that $G$ lifts to $\widetilde{X}$ by their actions on $K_X$. 
We denote by $G|_{\widetilde{X}}$ the group to which $G$ lifts, hence there exists a surjective group homomorphism $G|_{\widetilde{X}}\to G$ with the kernel being the Galois group $\textup{Gal}(\widetilde{X}/X)$ (cf. \cite[\S 2.15]{zhang2013algebraic}).
By \cite[Lemma 2.4]{zhang2013algebraic}, there exists a finite index subgroup $G_0$ of $G|_{\widetilde{X}}$ such that $(\widetilde{X},G_0)$ satisfies \textbf{Hyp\,(A)}.
Then, under the Albanese closure map  $\widetilde{X}'\rightarrow\widetilde{X}$ in Proposition \ref{pro1}, $G|_{\widetilde{X}}$ lifts to $\widetilde{X}'$ by the uniqueness. 
By \cite[Lemma 2.4 and \S 2.15]{zhang2013algebraic} again,  there exists a finite index subgroup $G_1$ of $G|_{\widetilde{X}'}$ such that $(\widetilde{X}',G_1)$ satisfies \textbf{Hyp\,(A)}.

From now on, we assume that $q^\circ(X)=q(X)$ and $K_X\sim 0$. If $q(X)=0$, then by Proposition \ref{keyproposition} (3), $X$ is projective and thus a $Q$-torus (cf.~\cite[Theorem 1.1]{zhang2016ndimension}), a contradiction to the equality $q(X)=q^\circ(X)$. 
So we may assume that $q^\circ(X)=q(X)>0$. Let $T:=\textup{Alb}(X)$ and $\alpha:X\rightarrow T$ the Albanese map. 
Thanks to Theorem \ref{thm-graf-thm1.10}, $\alpha$ is surjective holomorphic with connected fibres, and there exists an \'etale cover $T_1\rightarrow T$ such that $X\times_T T_1\cong F\times T_1$ for a connected fibre $F$ of $\alpha$ such that $\kappa(F)=q^\circ(F)=0$. So $\dim F=0$ or $2$. 

If $\dim F=0$,  then $\alpha$ is  bimeromorphic. 
So we have $0\equiv K_X=\alpha^*K_T+E\equiv E$, 
where the support of $E$ equals the exceptional locus of $\alpha$. Since $T$ is smooth, our $E\ge 0$. 
Hence, the above equation gives us $E=0$ and thus $X\cong T$.

If $\dim F=2$, then $T$ is an elliptic curve and thus projective.  Also, $G\subseteq\textup{Aut}(X)$ descends to a well-defined automorphism group $G_T$ of $T$. 
Moreover, $T_1\rightarrow T$ is an isogeny and then with $T_1$ replaced by a further isogeny $T_2\rightarrow T_1$, we may assume that $\theta:T_2\rightarrow T$ is a multiplication map. By \cite[Lemma 4.9]{nakayama2009building}, $G_{T}$ lifts to an automorphism group $G_{T_2}$ of $T_2$. 
See the following commutative diagram:
\[\xymatrix{
F\times T_2\ar[r]\ar[d]&T_2\ar[d]\\
F\times T_1\ar[r]\ar[d]&T_1\ar[d]\\
X\ar[r]&T
}\]
For each $g\in G$ on $X$, we get an induced automorphism $g|_T\in G_T$ and then $g|_T$ lifts to $g'\in\textup{Aut}(T_2)$ (which is not unique) on $T_2$. Here, there are $\deg\theta$  of choices for the lifting of  $g_T$ to $\textup{Aut}(T_2)$ (cf.~\cite[Proof of Lemma 4.9]{nakayama2009building}).

 Since $F\times T_2\cong X\times_T T_2\subseteq X\times T_2$ with respect to $\alpha:X\rightarrow T$ and $\theta: T_2\rightarrow T$, we restrict $g\times g':X\times T_2\rightarrow X\times T_2$ to $F\times T_2$, for $g\in G$ and $g'\in G_{T_2}$ such that $\alpha\circ g=g_T\circ\alpha$ and $g_T\circ\theta=\theta\circ g'$. 
 Then, we get an induced subgroup $G_{F\times T_2}\subseteq\textup{Aut}(F\times T_2)$  (consisting of all liftings of $G$) and there is a natural surjective group homomorphism $G_{F\times T_2}\rightarrow G$.
By \cite[Lemma 2.4]{zhang2013algebraic}, there exists a finite index subgroup $G_0\subseteq G_{F\times T_2}$ such that  $(F\times T_2,G_0)$ satisfies \textbf{Hyp\,(A)}. 
So we get a non-trivial $G_0$-equivariant fibration $F\times T_2\rightarrow \textup{Alb}(F\times T_2)=:T_2$, a contradiction to Lemma \ref{lem-zhang-2-10}. 
Thus, $X$ is a $Q$-torus. 
\end{proof}

\begin{proof}[Proof of Proposition \ref{thm1.2}] 
By Proposition \ref{prop_kappa<=0}, we have $\kappa(X)=0$.
Hence, $K_X\equiv 0$ (cf.~\cite[Theorem 1.1]{campana2016abundance}), and our proposition follows from  Lemma \ref{smallcase1}. 
\end{proof}

\section{The proof of Theorem \ref{thm1.1}}\label{section3}
In this section, we  prove our main result.
We begin with the lemma below, which states the splitting property for  specific  bimeromorphic transformations on product spaces. 
\setlength\parskip{0pt}
\begin{lem}\label{splitting lemma}
Let $T$ be a complex torus and $F$ a  (minimal) K3 surface with $q^\circ(F)=0$. 
Then every  bimeromorphic transformation $g:F\times T\dashrightarrow F\times T$ is an automorphism and splits into the form $(\sigma,\tau)$ such that $\sigma\in\textup{Aut}(F)$ and $\tau\in\textup{Aut}(T)$.
\end{lem}

\begin{proof}
First, we consider the following commutative diagram \[\xymatrix{F\times T\ar@{-->}[r]^g\ar[d]_{\alpha}&F\times T\ar[d]^{\alpha}\\
T\ar[r]_{\tau}&T
}\]
which is induced by the Albanese map $\alpha:F\times T\rightarrow \textup{Alb}(F\times T)=T$ (cf.~\cite[Corollary 2.11]{nakayama2009building}) 
with the descended  $\tau:T\rightarrow T$ being an automorphism (cf.~\cite[Lemma 9.11]{ueno2006classification}).
 By the diagram, for a general point $t\in T$, there exists a bimeromorphic map $\sigma_t:F\dashrightarrow F$. In other words, there exists a meromorphic map $\chi:T\dashrightarrow \textup{Bim}(F)$. Since $F$ is a minimal smooth surface, $\textup{Bim}(F)=\textup{Aut}(F)$. Further, by the property of K3 surfaces, $h^0(F,T_F)=h^0(F,\Omega_F^1)=q(F)=0$ (where $T_F\cong\Omega_F^1\otimes(\Omega_F^2)^\vee\cong\Omega_F^1$ is the tangent sheaf). So $\textup{Bim}(F)=\textup{Aut}(F)$ is discrete. As a result, $\chi$ is  constant, i.e., the automorphism $\sigma_t$ is independent of the choice of $t\in T$, which completes our proof. 
\end{proof}


\setlength\parskip{4pt}
\par \vskip 0.3pc \noindent
\textbf{Proof of Theorem \ref{thm1.1}:}\setlength\parskip{0pt}
First, we note that $\kappa(X)\le 0$ (cf.~Proposition \ref{prop_kappa<=0}).

(1). Since  $X$ has at worst terminal singularities and $K_X$ is not pseudo-effective, our $X$ is uniruled. 
Taking a $G$-equivariant resolution $\widetilde{X}\rightarrow X$ (cf.~\cite[Theorem 2.0.1]{wlodarczk2009resolution}), we may assume that $X$ is smooth. 
Let $\psi:X\dashrightarrow C(X)$ be the meromorphic map to the cycle space of $X$ which defines the maximal rationally connected fibration of $X$, sending a general point $x\in X$ to the maximal rationally connected subvariety containing $x$.
Since each $g\in G$ is an automorphism and hence proper, it follows from  \cite[Chapter IV \S2 Th\'eor\`eme 6]{barlet1975easpace} that it descends to a self-morphism $g_c$ on $C(X)$ by the push-forward operation $g_*$. 
Since $g$ is an automorphism, $g_*(M)$ is reduced for every rationally connected submanifold $M$, hence  $\psi\circ g=g_c\circ\psi$. 

Let $X\dashrightarrow Z\to C(X)$ be the Stein factorization of $\psi$  with $Z\to C(X)$ being a finite morphism (cf.~\cite[Definition 3.2]{nakayama2009building}). 
Since $\psi\circ g=g_c\circ\psi$,  the (irreducible) graph $\Gamma_{xc}$ of the meromorphic map $X\dashrightarrow C(X)$ is  $(g\times g_c)$-invariant.
Hence, we get an endomorphism $g_{\Gamma_{xc}}:=(g\times g_c)|_{\Gamma_{xc}}$.
Let $\widetilde{\Gamma_{xc}}$ be the normalization of $\Gamma_{xc}$. 
Then $g_{\Gamma_{xc}}$ lifts to $g|_{\widetilde{\Gamma_{xc}}}$, and $Z\to C(X)$ is given by the usual Stein factorization of $p_2:\widetilde{\Gamma_{xc}}\xrightarrow{\phi_1} Z\xrightarrow{\phi_2} C(X)$ (cf.~\cite[Remark after Definition 3.2]{nakayama2009building}).
Since   $p_2\circ g|_{\widetilde{\Gamma_{xc}}}=g_c\circ p_2$,
 for every (connected) fibre $F$ of $\phi_1$, the image $g|_{\widetilde{\Gamma_{xc}}}(F)$ is contracted by $p_2$, and hence $\phi_1(g|_{\widetilde{\Gamma_{xc}}}(F))$ is a point.
By the rigidity lemma (cf.~\cite[Lemma 4.1.13]{beltrametti1995the}), $\phi_1\circ g|_{\widetilde{\Gamma_{xc}}}$ factors  through $\phi_1$, i.e., $g|_{\widetilde{\Gamma_{xc}}}$ (and hence $g$) descends to an automorphism $g_z$ on $Z$.
Furthermore,  since general maximal rationally connected subvariety of $X$ is contracted along $\psi$, our $Z$ is not uniruled. 

Denote by $\Gamma_{xz}\subseteq X\times Z$ the graph of the almost holomorphic fibration $X\dashrightarrow Z$.
Then we get an endomorphism $g_{\Gamma_{xz}}:=(g\times g_z)|_{\Gamma_{xz}}$, noting that the graph $\Gamma_{xz}$ is $(g\times g_z)$-invariant. 
Taking the normalization $X'$ of $\Gamma$, we see that there is an induced endomorphism $g'$ on $X'$.
Let $p:X'\to X$ be the induced projection.
Then  $g\circ p=p\circ g'$ and we can lift $g$ to an automorphism $g'$ on $X'$.
Hence, there is an induced automorphism group $G|_{X'}$ of rank $2$ on $X'$.
Replacing $X'$ by a $G$-equivariant resolution (cf.~\cite[Section 1.4]{nakayama2009building} and \cite[Theorem 2.0.1]{wlodarczk2009resolution}) if necessary, we may assume that $X'$ is smooth.
Now, we get a $G$-equivariant fibration $X'\to Z$.
By Lemma \ref{lem-zhang-2-10}, $Z$ can only be a single point, which implies that $X$ is rationally connected.
So Theorem \ref{thm1.1} (1) is proved.

(2). Since $X$ has at worst terminal singularities and $K_X$ is pseudo-effective, by \cite[Theorem 1.1]{horing2016minimal}, there exists a minimal model $X_{\textup{min}}$ of $X$, i.e., there is a bimeromorphic map which is a composite of divisorial contractions and flips:
$X\dashrightarrow X_{\textup{min}}$ such that $K_{X_\textup{min}}$ is nef. 
 Then, the automorphism group $G$ descends to a bimeromorphic transformation subgroup $G|_{X_{\textup{min}}}\subseteq\textup{Bim}(X_{\textup{min}})$ of rank $2$. 
By Proposition \ref{prop_kappa<=0}, $\kappa(X)=\kappa(X_{\textup{min}})=0$.
Thus, $K_{X_{\textup{min}}}\sim_{\mathbb{Q}} 0$ by the abundance (cf.~\cite[Theorem 1.1]{campana2016abundance}).

In the following, we aim to construct a quasi-\'etale splitting cover: $\widetilde{X_{\textup{min}}}:=F\times T\rightarrow X_{\textup{min}}$ 
  such that $F$ is a weak Calabi-Yau space (and thus $\dim F=0$ or $2$), $T$ is a torus and $G|_{X_\textup{min}}$ lifts to a subgroup $G_{\widetilde{X_{\textup{min}}}}\subseteq\textup{Aut}(\widetilde{X_{\textup{min}
  }})$.
  
  By Proposition \ref{keyproposition} and \cite[Theorem 1.1]{zhang2016ndimension}, the result holds  when  $q^\circ(X_{\textup{min}})=0$.  
  We prepare to exclude the case   $q^\circ(X_{\textup{min}})=1$ (cf.~\textbf{Step 1-3}) so that the remaining case  $q^\circ(X_{\textup{min}})=3$ is what we desire. Suppose the contrary that $q^\circ(X_{\textup{min}})=1$.
 
\par \vskip 0.3pc \noindent
\textbf{Step 1.} We reduce to the Albanese closure of $X_{\textup{min}}$. Fixing an element $g\in G|_{X_\textup{min}}\subseteq\textup{Bim}(X_{\textup{min}})$ and taking the Albanese closure $\tau:X'_{\textup{min}}\rightarrow X_{\textup{min}}$, we get a meromorphic map $X'_{\textup{min}}\rightarrow X_{\textup{min}}\dashrightarrow X_{\textup{min}}$. 
After the Stein factorization (cf.~\cite[Definition 3.2]{nakayama2009building}), we get:
\[\xymatrix{X_{\textup{min}}'\ar[dr]_\tau&X_1\ar[l]_\sigma\ar[d]^{\tau'}&X'_{\textup{min}}\ar@{-->}[l]_{g'}\ar[d]^\tau\\
&X_{\textup{min}}&X_{\textup{min}}\ar@{-->}[l]^g}\]
where $\tau'$ is a finite morphism and $g'$ is bimeromorphic. 
Since $g$ is  isomorphic in codimension one (cf.~Theorem \ref{lemisomorphic}),  by the property of Stein factorization and the  quasi-\'etaleness of $\tau$,  our $g'$ is also isomorphic in codimension one and  $\tau'$ is quasi-\'etale. 
Moreover, $$q^\circ(X_1)=q^\circ(X_{\textup{min}})=q(X'_{\textup{min}})=q(X_1)=1.$$ 
By Definition \ref{defi-alb-closure} (4) and Proposition \ref{pro1}, $\tau'$ factors through $\tau$, i.e., there exists a surjective morphism $\sigma:X_1\rightarrow X'_{\textup{min}}$ such that $\tau'=\tau\circ\sigma$. 
Since $\deg\tau=\deg\tau'$, our $\sigma$ is bimeromorphic and we see that $g\in G|_{X_\textup{min}}$ lifts to a bimeromorphic transformation on its Albanese closure. Note that $G|_{X_{\textup{min}}}$ lifts to $G_{X_{\textup{min}}'}$ such that its quotient is $G|_{X_{\textup{min}}}$. 

\par \vskip 0.3pc \noindent
\textbf{Step 2.} From now on, we may assume that $X_{\textup{min}}$ has at worst terminal singularities, $q^\circ(X_{\textup{min}})=q(X_{\textup{min}})=1$ and $K_{X_{\textup{min}}}\sim_{\mathbb{Q}}0$. Let $\alpha: X_{\textup{min}}\rightarrow \textup{Alb}(X_{\textup{min}})=:E$ denote the Albanese map. 
By Theorem \ref{thm-graf-thm1.10}, there exists an \'etale cover $E_1\rightarrow E$ such that $X_{\textup{min}}\times_E E_1\cong F\times E_1$ for a fibre $F$ of $\alpha$. 
By the choice of the Albanese closure, $q^\circ(F)=0$ and $\dim F=2$. 
With $F$ further replaced by its global index-one cover, we may assume that $K_F\sim 0$. 
Also, the subgroup $G|_{X_{\textup{min}}}\subseteq\textup{Bim}(X_{\textup{min}})$ descends to $G_E\subseteq \textup{Aut}(E)$ on $E$ (cf.~\cite[Lemma 9.11]{ueno2006classification}). 
Since $E_1\rightarrow E$ is an isogeny, there exists another isogeny  $E_2\rightarrow E_1$ such that the composite map $\theta: E_2\rightarrow E$ is a multiplication (finite) map. 
Then, by \cite[Lemma 4.9]{nakayama2009building},  $G_E$ lifts to $G_{E_2}$ such that $G_{E_2}\rightarrow G_E$ is surjective.
Moreover, after the base change $E_2\rightarrow E_1$, $X_{\textup{min}}\times_E E_2\cong F\times E_2$ with respect to $\alpha$ and $\theta$. Therefore, regarding $F\times E_2$ as a subset of $X_{\textup{min}}\times E_2$, we get a natural lifting $G_{F\times E_2}\subseteq \textup{Bim}(F\times E_2)$ by restricting $g\times g'$ on $X_{\textup{min}}\times E_2$ to $F\times E_2$ for each $g\in G|_{X_{\textup{min}}}$ and $g'\in G_{E_2}$ such that $g_E\circ \alpha=\alpha\circ g$ and $\theta\circ g_E=g'\circ\theta$. Here, $g\in G|_{X_{\textup{min}}}$  descends to $g_E$ on $E$ and lifts to $g'$ (which is not  unique) on $G_{E_2}$. By the proof in Lemma \ref{smallcase1}, there exists a finite index subgroup $G_1$ of $G_{F\times E_2}$ such that $G_1$ is free of rank $2$ (cf.~\cite[Lemma 2.4]{zhang2013algebraic}).

\par \vskip 0.3pc \noindent
\textbf{Step 3.} Taking a minimal resolution $\widetilde{F}\rightarrow F$, we see that $\widetilde{F}$ is a minimal K3 surface. 
Then we get a natural lifting $\widetilde{G_1}\subseteq\textup{Bim}(\widetilde{F}\times E_2)$ of $G_1$ on $F\times E_2$. 
By Lemma \ref{splitting lemma}, $$\textup{Bim}(\widetilde{F}\times E_2)=\textup{Aut}(\widetilde{F}\times E_2)\subseteq \textup{Aut}(\widetilde{F})\times\textup{Aut}(E_2).$$
Hence,  $(\widetilde{F}\times E_2, \widetilde{G_1})$ satisfies \textbf{Hyp\,(A)} (cf.~\cite[Theorem 1.1]{dinh2012on}). 
As a result, there exists a $\widetilde{G_1}$-equivariant non-trivial fibration $\widetilde{F}\times E_2\rightarrow E_2$, contradicting Lemma \ref{lem-zhang-2-10}. 
Therefore, the case $q^\circ(X)=1$ cannot happen. 
So $q^\circ(X_{\textup{min}})=3$ and $X_{\textup{min}}$ is a $Q$-torus.

\par \vskip 0.3pc \noindent
\textbf{Step 4.} So far we have  completed the proof of the first part of Theorem \ref{thm1.1} (2). 
Since $X_{\textup{min}}$ is a $Q$-torus, its Albanese closure $X_{\textup{min}}'$ satisfies the following properties: $K_{X_\textup{min}'}\sim_{\mathbb{Q}} 0$, $q(X_{\textup{min}}')=3$, $X_{\textup{min}}'$ has at worst terminal singularities and the  subgroup $G|_{X_{\textup{min}}}\subseteq\textup{Bim}(X_{\textup{min}})$ lifts to $G_{X_{\textup{min}}'}\subseteq\textup{Bim}(X_{\textup{min}}')$ as in \textbf{Step 1} such that a finite index subgroup of $G_{X_{\textup{min}}'}$ has rank $2$  (cf.~\cite[Lemma 2.4]{zhang2013algebraic}). 
With the same proof as in Lemma \ref{smallcase1}, $X_{\textup{min}}'$ is isomorphic to a complex $3$-torus by showing that the Albanese map $X_{\textup{min}}'\rightarrow \textup{Alb}(X_{\textup{min}}')$ is an isomorphism. Thus, $G_{X_{\textup{min}}'}\subseteq\textup{Bim}(X_{\textup{min}}')=\textup{Aut}(X_{\textup{min}}')$.


\par \vskip 0.3pc \noindent
\textbf{Step 5.} In this step, we shall prove that each $g\in G|_{X_{\textup{min}}}$ is actually an automorphism. 
Let $\pi:T_1:=X_{\textup{min}}'\rightarrow X_{\textup{min}}$ be the Albanese closure of $X_{\textup{min}}$ as in  \textbf{Step 4}. 
Then, $G|_{X_{\textup{min}}}$ lifts to $G_{T_1}\subseteq\textup{Aut}(T_1)$ as in \textbf{Step 1}. 
Consider the following commutative diagram, 
\[
\xymatrix{T_1 \ar[r]^{g|_{T_1}}\ar[d]_\pi &T_1 \ar[d]^\pi\\
X_{\textup{min}}\ar@{-->}[r]^g&X_{\textup{min}}
}
\]
where $g\in G|_{X_{\textup{min}}}$ lifts to $g|_{T_1}$ (which is not unique) on $T_1$. 
Let $H:=\textup{Gal}(T_1/X_{\textup{min}})$. 
Since each $g|_{T_1}$ normalizes  $H$, i.e., $g|_{T_1} Hg|_{T_1}^{-1}\subseteq H$ for any $g|_{T_1}\in G_{T_1}$, 
 the composite map $\pi\circ g|_{T_1}$ is $H$-invariant. 
 By the universality of the quotient morphism $\pi$ over the \'etale loci, the composite $\pi\circ g|_{T_1}$ factors through $\pi$. Thus, $g|_{T_1}$ descends to a well-defined morphism 
$\bar{g}:X_{\textup{min}}\rightarrow X_{\textup{min}}$, acting biregularly on $X_{\textup{min}}$. 
Since the meromorphic map $g$ is determined by an open dense subset $U\subseteq X_{\textup{min}}$ and $g|_U=\bar{g}|_U$, our $g=\bar{g}$ is biholomorphic.

\begin{proof}[\textup{\textbf{End of Proof of Theorem \ref{thm1.1}}}]
Now,  each $g\in G|_{X_\textup{min}}$ is  a (biholomorphic) automorphism and $(X_{\textup{min}}, G|_{X_{\textup{min}}})$ satisfies \textbf{Hyp\,(A)}. 
Thanks to \cite[Lemma 2.4 and \S 2.16]{zhang2013algebraic}, $X_{\textup{min}}\cong T_1/H$ for a finite group $H$ acting freely outside a finite set of a complex $3$-torus $T_1$. We complete the proof of Theorem \ref{thm1.1}.
\end{proof}

We end up  with the following extended result,  
the conditions of which coincide with those in \cite[Theorem 1.1]{zhang2009theorem}.
We say that a group $G$ is \textit{virtually unipotent} (resp.\,\textit{virtually solvable}), if a finite-index subgroup $G_1$ of $G$  is unipotent (resp.\,solvable). 
\begin{remark}[\textbf{A further generalization}]\label{rmk_solvable}
\textup{Let $X$ be a  compact K\"ahler (smooth) threefold, $G\subseteq \textup{Aut}(X)$ a subgroup such that  $G|_{H^{1,1}(X,\mathbb{C})}$ is solvable and has connected Zariski-closure ($Z$-connected for short), and $G/N(G)\cong\mathbb{Z}^{\oplus 2}$ (this is the case when $(X,G)$ satisfies \textbf{Hyp\,(A)}).
Here, $N(G)\subseteq G$ is the set of elements $g\in G$ of null entropy.
Then with the same strategy for the proof of Theorem \ref{thm1.1}, we can show that either $X$ is rationally connected, or $X$ is $G$-equivariantly  bimeromorphic to a $Q$-torus.}

\textup{First, \cite[Lemma 2.10]{zhang2009theorem} is still valid in this situation.
Second, for any equivariant finite \'etale Galois cover $\pi:(\widetilde{X},\widetilde{G})\to (X,G)$ with $G\cong\widetilde{G}/\textup{Gal}(\widetilde{X}/X)$, we have a natural surjection $\widetilde{G}|_{H^{1,1}(\widetilde{X},\mathbb{C})}\twoheadrightarrow \widetilde{G}|_{\pi^*H^{1,1}(X,\mathbb{C})}\cong G|_{H^{1,1}(X,\mathbb{C})}$.
Hence, by considering the virtually unipotent kernel (cf. \cite[Theorem 2.2]{campana2014automorphisms}), if $G|_{H^{1,1}(X,\mathbb{C})}$ is virtually solvable and $Z$-connected, then so is $\widetilde{G}|_{H^{1,1}(\widetilde{X},\mathbb{C})}$ (\cite[Lemma 5.5]{dinh2012tits}).
Moreover, note that $N(\widetilde{G})=\pi^{-1}(N(G))$  (cf.~\cite[Appendix A, Lemma A.8]{nakayama2009building}), hence $\pi$ induces an isomorphism $\widetilde{G}/N(\widetilde{G})\cong G/N(G)\cong\mathbb{Z}^{\oplus 2}$.}

\textup{For terminal threefolds (as the end product of the MMP), the same arguments work  after we replace $H^{1,1}(X,\mathbb{C})$ by the complexified Bott-Chern cohomology $H^{1,1}_{\textup{BC}}(X)\otimes_{\mathbb{R}}\mathbb{C}$.
Indeed, for a $G$-equivariant generically finite dominant meromorphic map $\pi:V\dashrightarrow W$ between normal compact K\"ahler spaces with only  rational singularities, we can show that,  $G|_{H^{1,1}_{\textup{BC}}(V)\otimes\mathbb{C}}$ is $Z$-connected and virtually solvable if and only if so is $G|_{H^{1,1}_{\textup{BC}}(W)\otimes\mathbb{C}}$ (cf.~\cite[Lemma 2.7]{zhong2019intamplified} and \cite[Lemma 4.2 and its proof]{hu2020a}).
}	
\end{remark}

\begin{thm}\label{solvable_case}
Let $X$ be a compact K\"ahler threefold with at worst terminal singularities.	
Let $G\subseteq\textup{Aut}(X)$ be a subgroup such that $G|_{H^{1,1}_{\textup{BC}}(X)}$ is solvable and $Z$-connected, and $G/N(G)\cong\mathbb{Z}^{\oplus 2}$.
Then $X$ is either rationally connected or $G$-equivariantly bimeromorphic to a $Q$-torus.
\end{thm}

\begin{lem}
Let $\pi:V\dashrightarrow W$ be a $G$-equivariant generically finite dominant meromorphic map between normal compact K\"ahler spaces with only  rational singularities. 
Then  $G|_{H^{1,1}_{\textup{BC}}(V)\otimes\mathbb{C}}$ is $Z$-connected and virtually solvable if and only if so is $G|_{H^{1,1}_{\textup{BC}}(W)\otimes\mathbb{C}}$.	
\end{lem}

\begin{proof}
First we consider the case when $\pi$ is holomorphic.
It is known that $\pi^*:\pi^*H^{1,1}_{\textup{BC}}(W)\otimes\mathbb{C}\to \pi^*H^{1,1}_{\textup{BC}}(V)\otimes\mathbb{C}$ is injective (cf.~\cite[Proposition 2.8]{zhong2019intamplified}) and $\pi^*H^{1,1}_{\textup{BC}}(W)\otimes\mathbb{C}$ is $G$-invariant via the pullback action.
The `only if' direction is easy.
Also, being $Z$-connected is automatically true after replacing $G$ by a finite-index subgroup.
So it suffices to show that if $G|_{H^{1,1}_{\textup{BC}}(W)\otimes\mathbb{C}}$ is virtually solvable, then $G|_{H^{1,1}_{\textup{BC}}(V)\otimes\mathbb{C}}$ is virtually solvable. 
Consider the natural restriction group homomorphism $G|_{H^{1,1}_{\textup{BC}}(V)\otimes\mathbb{C}}\twoheadrightarrow G|_{\pi^*H^{1,1}_{\textup{BC}}(W)\otimes\mathbb{C}}\cong G|_{H^{1,1}_{\textup{BC}}(W)\otimes\mathbb{C}}$. 
Let $K|_{H^{1,1}_{\textup{BC}}(V)\otimes\mathbb{C}}$ denote its kernel, where $K\subseteq G$ is a subgroup of $G$.
Choose any K\"ahler class $\xi_Y$ on $Y$.
Then $K$ fixed the nef and big class $\xi_Y$, since $\pi$ is generically finite (cf.~\cite[Proposition 2.6]{zhong2019intamplified}).
Therefore, $K\subseteq N(G)$ is of null entropy (cf.~\cite[Corollary 2.2]{dinh2015compact}).
So by \cite[Theorem 2.2]{campana2014automorphisms}, $K|_{H^{1,1}_{\textup{BC}}(V)\otimes\mathbb{C}}$is virtually unipotent and hence virtually solvable.
Then it follows from \cite[Lemma 5.5]{dinh2012tits} that $G|_{H^{1,1}_{\textup{BC}}(V)\otimes\mathbb{C}}$ is virtually solvable.

Now we reduce the meromorphic map to the holomorphic case.
Taking a $G$-equivariant resolution of the normalization of the graph of $\pi$, we get the natural projections $p_1:X\to V$ and $p_2:X\to W$.
Then $p_1$ is bimeromorphic and $p_2$ is generically finite.
Our lemma is proved. 
\end{proof}

\begin{proof}[Proof of Theorem \ref{solvable_case}]
Almost the same as the proof of Theorem \ref{thm1.1}.
We only need to consider \textbf{Step 3} in the proof to show $G|_{E\times\widetilde{F}}$ is $Z$-connected and virtually solvable (it automorphically has rank two).
So we got the generically finite dominant meromorphic map
$\tau:E\times\widetilde{F}\to X_{\textup{min}}$ followed by $X_{\textup{min}}\to X$.
Applying our key lemma, we have $G|_{E\times\widetilde{F}}$ is $Z$-connected and virtually solvable.
Then, applying \cite[Lemma 2.10]{zhang2009theorem}, we can exclude this case and hence $X$ is a $Q$-torus.
Other parts are the same as Proof of Theorem \ref{thm1.1}.
\end{proof}

\end{document}